\documentclass[12pt]{scrartcl}

\usepackage[]{amsmath, amssymb,amsfonts,amsthm,mathtools,braket,color,enumerate,stmaryrd}

\usepackage{hhline}
\usepackage{url}
\usepackage{here}
\usepackage{tikz}
\usepackage{tikz-cd}
\usepackage{hyperref}
\usepackage {threeparttable}
\usepackage{multirow}
\usepackage{makecell}
\usepackage{cellspace}
\usepackage{float}
\hypersetup{
  colorlinks   = true, 
  urlcolor     = blue, 
  linkcolor    = blue, 
  citecolor   = red 
}
\usetikzlibrary{decorations}
\usetikzlibrary{arrows}
\tikzstyle{v} = [circle, draw, inner sep=2pt, minimum size=3pt, fill=black]
\tikzstyle{l} = [rectangle, draw, rounded corners]

\theoremstyle{plain}
\newtheorem{theorem}{Theorem}[section]
\newtheorem{lemma}[theorem]{Lemma}
\newtheorem{proposition}[theorem]{Proposition}
\newtheorem{corollary}[theorem]{Corollary}

\theoremstyle{remark}

\theoremstyle{definition}
\newtheorem{definition}[theorem]{Definition}

\newtheorem{example}[theorem]{Example}

\newtheorem{remark}[theorem]{Remark}

\begin{document}

\title{Geometric characterization of $p$-exceptional monomial GAPN functions}
\author{
Masamichi Kuroda
\thanks{
Faculty of Engineering, Nippon Bunri University, Oita 870-0316, Japan. 
\\
Email: kurodamm@nbu.ac.jp
} \and
Kentaro Mitsui
\thanks{
Department of Mathematical Sciences, Faculty of Science, University of the Ryukyus, Okinawa 903-0213, Japan.
\\ 
Email: mitsui@math.u-ryukyu.ac.jp}
}

\date{}

\maketitle

\begin{abstract}
On finite fields of characteristic $p$, PN (perfect nonlinear) functions for odd $p$ and APN (almost perfect nonlinear) functions for even $p$ are well-known classes of highly nonlinear functions. 
GAPN (generalized almost perfect nonlinear) functions were introduced as a generalization of APN functions for even $p$ to all $p$. 
One of the main targets of studies on such highly nonlinear functions is their classification. 
While $p$-exceptional monomial PN and $2$-exceptional monomial APN functions have been classified, the corresponding problem for GAPN functions remains open. 
Here, a polynomial over $\mathbb{F}_p$ is called a $p$-exceptional PN (resp.\ APN, resp.\ GAPN) function if it is a PN (resp.\ APN, resp.\ GAPN) function on $\mathbb{F}_{p^n}$ for infinitely many positive integers $n$. 

In this paper, we give a geometric characterization of $p$-exceptional monomial GAPN functions. 
To this end, for a prime power $q$, we establish a geometric characterization of non-zero polynomials in $\mathbb{F}_{q}[x,y]$ having no absolutely irreducible factor defined over $\mathbb{F}_{q}$. 
\end{abstract}



\tableofcontents

\section{Introduction} \label{sec: intro}

Let $p$ be a prime number and $n$ be a positive integer. 
We choose an algebraic closure $\overline{\mathbb{F}_p}$ of the finite field $\mathbb{F}_p$ of order $p$. 
Let $\mathbb{F}_{p^n} \subset \overline{\mathbb{F}_p}$ denote the finite field of order $p^n$. 

\begin{definition}
A function $f \colon \mathbb{F}_{p^n} \to \mathbb{F}_{p^n}$ is called a \textit{perfect nonlinear} (\textit{PN}) (resp.\ \textit{almost perfect nonlinear} (\textit{APN})) function if 
\begin{align*}
\# \Set{ x \in \mathbb{F}_{p^n} | D_{a} f(x) \coloneqq f(x+a) - f(x) = b} \leq 1 \ \mbox{(resp.\ $2$)} 
\end{align*}
for any $a \in \mathbb{F}_{p^n}^{\times}$ and any $b \in \mathbb{F}_{p^n}$. 
\end{definition}

PN and APN functions have been studied in wide mathematical areas such as finite geometry, coding theory and cryptography. 
In the case $p=2$, the equality 
\begin{align} \label{algebraic condition:p=2}
D_{a} f(x+a) = D_{a} f(x) \quad (a \in \mathbb{F}_{2^n}^{\times}) 
\end{align}
implies the non-existence of PN functions. 
Since APN functions for odd $p$ do not satisfy the equality (\ref{algebraic condition:p=2}) in general, their properties are quite different from those in the case $p=2$. 
APN functions for even $p$ were generalized to all $p$ in the following way~\cite{KT2017}: 

\begin{definition} \label{def: GAPN}
A function $f \colon \mathbb{F}_{p^n} \to \mathbb{F}_{p^n}$ is called a \textit{generalized almost perfect nonlinear} (\textit{GAPN}) function if 
\begin{align*}
\# \Set{ x \in \mathbb{F}_{p^n} | GD_{a} f(x) \coloneqq \sum_{j \in \mathbb{F}_p} f(x+ja) = b} \leq p  
\end{align*}
for any $a \in \mathbb{F}_{p^n}^{\times}$ and any $b \in \mathbb{F}_{p^n}$. 
\end{definition}

Note that in the case $n=1$, any function $f \colon \mathbb{F}_{p} \to \mathbb{F}_{p}$ is a GAPN function on $\mathbb{F}_{p}$. 
In the case $p=2$, GAPN functions coincide with APN functions. 
In addition, the equality 
\begin{align*} 
GD_{a} f(x+a) = GD_{a} f(x) \quad (a \in \mathbb{F}_{p^n}^{\times}) 
\end{align*}
holds, which is an analogue of the equality (\ref{algebraic condition:p=2}). 

One of the main targets of studies on these highly nonlinear functions is their classification. 
Such a classification is difficult for general polynomials and has been studied mainly for monomials. 
Conjectures are posed for the classifications of monomial PN and APN functions, but they remain unsolved. 
On the other hand, restricting these functions to the class $p$-exceptional (defined below) allows us to classify $p$-exceptional monomial PN and $2$-exceptional monomial APN functions (see Theorems \ref{thm: 2-excep}--\ref{thm: PN functions}). 
The notion of $p$-exceptionality is also important in several related areas. 
For example, the classification of $2$-exceptional monomial APN functions yields a classification of exceptional exponents in coding theory (see \cite{HM2011} for more details).

\begin{definition} \label{def:exceptional}
A polynomial $f \in \mathbb{F}_p[x]$ is called a \textit{$p$-exceptional PN} (resp.\ \textit{APN}, resp.\ \textit{GAPN}) function (over $\mathbb{F}_p$) if $f$ is a PN (resp.\ APN, resp.\ GAPN) function on $\mathbb{F}_{p^n}$ for infinitely many positive integers $n$. 
\end{definition}

The following theorem was conjectured in~\cite{JMW1995} based on the earlier work in~\cite{JW1993}. 
It was subsequently studied by several researchers, including Jedlicka~\cite{Jedlicka2007}, and finally proved by Hernando and McGuire~\cite{HM2011}.

\begin{theorem} \label{thm: 2-excep}
The monomial $f_d(x) = x^d$ is a $2$-exceptional APN function if and only if 
\begin{enumerate}[(i)] 
\item 
$d = 2^i + 2^j$, where $i > j \geq 0$, or 
\item 
$d = 2^{2i+j} - 2^{i+j} + 2^{j}$, where $i > 0$ and $j \geq 0$. 
\end{enumerate}
\end{theorem}

Similarly to Theorem \ref{thm: 2-excep}, Hernando, McGuire and Monserrat conjectured the classification of $p$-exceptional monomial PN functions in~\cite{HMM2014}, which was settled by Leducq~\cite{Leducq2015} and Zieve~\cite{Zieve2015}. 

\begin{theorem} \label{thm: PN functions}
The monomial $f_d(x) = x^d$ is a $p$-exceptional PN function if and only if 
\begin{enumerate}[(i)] 
\item 
$d = p^i + p^j$, where $p \geq 3$ and $i \geq j \geq 0$, or 
\item 
$d = ( 3^i + 3^j )/2$, where $p=3$, $i > j \geq 0$ and $i \not \equiv j \mod 2$.
\end{enumerate}
\end{theorem}

These results naturally lead us to study the classification of $p$-exceptional monomial GAPN functions. 
Table~\ref{table:p-exceptional GAPN} below lists the known series of such functions for odd $p$. 

\begin{table}[H]
\centering
\caption{The known series of $p$-exceptional monomial GAPN functions $f_d (x) = x^d$ for $p \geq 3$. 
It suffices to consider exponents $d$ with $d\not\equiv0\mod p$ (see Remark~\ref{d not equiv 0 mod p}).
Let $d^{\circ} (f_d)$ denote the algebraic degree of $f_d$. 
Put $[a, b] \coloneqq \{ m \in \mathbb{Z} \mid a \leq m \leq b \}$ for each $a \in \mathbb{Z}$ and each $b \in \mathbb{Z}$. 
The relationships between the families in this table and the known results are discussed in Remark~\ref{rem: table1} below.} 
\label{table:p-exceptional GAPN}
{\fontsize{11}{12}\selectfont
\begin{tabular}{|Sl|Sc|Sc|} \hline
\multicolumn{1}{|Sc|}{Exponents $d$} & $d^{\circ} (f_d)$ & Reference
\\ \hline \hline 
(1) $p^{i_1} + \cdots + p^{i_{p-1}} + 1$, where $i_1 \geq \cdots \geq i_{p-1} \geq 0$, $i_1 \ne 0$.
& $p$ & 
\cite[Prop.~9]{Kuroda2020}
\\ \hline
\makecell[l]{(2) \\~ } 
\makecell[l]{$k_2 p^{i} + k_1$, where $i > 0$, $1 \leq k_1$, $k_2 \leq p-1$, \\ 
$k_1 + k_2 \geq p$, $\gcd(k_1+k_2, p-1)=1$.} 
& in $[p, 2p-3]$ & 
\multirow{2}{*}[-8pt]{Prop.~\ref{ex: exceptionality}}
\\ \cline{1-2}
\makecell[l]{(3) \\~ }
\makecell[l]{
$k_2 p^{2 i} + (p-1) p^i + k_1$, where $i > 0$, $1 \leq k_1$, $k_2 \leq p-1$, \\
$k_1 + k_2 < p-1$, $\gcd(k_1+k_2, p-1)=1$.}  & in $[p+2, 2p-3]$ & 
\\
\hline
  \end{tabular}
}
\end{table}

In this paper, we study the $p$-exceptionality of a monomial GAPN function $f_d(x) = x^d$. 
In the proof of Theorem~\ref{thm: 2-excep}, which classifies $2$-exceptional monomial APN functions, Corollary~1 in \cite{Jedlicka2007} plays an important role. 
This corollary is formulated in terms of a polynomial $\psi_d \in \mathbb{F}_p[x,y]$ associated with the exponent $d$ (see Definition~\ref{def: psid}). 
For example, if $d \not \equiv 0 \mod p$, then 
\begin{align*}
\psi_d (x,y) = \dfrac{\sum_{j \in \mathbb{F}_p} \left( (x+j)^d - (y+j)^d \right)}{(x-y)^p - (x-y)}. 
\end{align*}
This corollary states that if the polynomial $\psi_d$ has an absolutely irreducible factor defined over $\mathbb{F}_2$, then the monomial $f_d$ is not a $2$-exceptional APN function. 
Jedlicka~\cite{Jedlicka2007} and Hernando--McGuire~\cite{HM2011} proved that if $\psi_d \ne 0$, then the polynomial $\psi_d$ has an absolutely irreducible factor defined over $\mathbb{F}_2$ for any exponent $d$ other than those listed in Theorem~\ref{thm: 2-excep}. 
Their result together with \cite[Corollary~1]{Jedlicka2007} leads to the classification stated in Theorem~\ref{thm: 2-excep}. 
As a consequence of Theorem~\ref{thm: 2-excep}, the converse of \cite[Corollary~1]{Jedlicka2007} also holds. 
That is, the monomial $f_d$ is a 2-exceptional APN function if and only if the polynomial $\psi_d$ is non-zero and has no absolutely irreducible factor defined over $\mathbb{F}_2$. 
Our first main result is the following theorem, which extends this geometric characterization from even $p$ to all $p$.

\begin{theorem}
\label{thm:geom. irred. comp. induces not exceptional}
Let $p$ be a prime number and $d$ be a positive integer. 
Then the monomial $f_d(x) = x^d$ is a $p$-exceptional GAPN function if and only if the polynomial $\psi_d$ is non-zero and has no absolutely irreducible factor defined over $\mathbb{F}_p$.
\end{theorem}

This theorem shows that the definition of GAPN functions preserves the geometric characterization of 2-exceptional monomial APN functions. 
Thus, the definition can be regarded as a natural extension of the notion of APN functions in even $p$ from the viewpoint of algebraic geometry. 
Theorem~\ref{thm:geom. irred. comp. induces not exceptional} follows as a corollary of a more general proposition, which is proved in Section~\ref{Proof of Thm. 1.6} (see Proposition~\ref{prop: generalization of key proposition}). 
This proposition also yields our second main result (see Theorem~\ref{thm: main 2}).
This theorem gives a geometric characterization of non-zero polynomials in $\mathbb{F}_q[x,y]$ having no absolutely irreducible factor defined over $\mathbb{F}_q$, where $q$ is a power of $p$. 
For clarity, we state the following simplified special case. 

\begin{corollary} \label{thm:geom. char. of red. poly.}
Let $q$ be a prime power. 
Let $h \in \mathbb{F}_q[x,y] \setminus \{ 0 \}$. 
Then the following are equivalent: 
\begin{enumerate}[(i)]
\item 
the polynomial $h$ has no absolutely irreducible factor defined over $\mathbb{F}_q$; 
\item 
there exist infinitely many positive integers $n$ such that $V(h) (\mathbb{F}_{q^n}) \subseteq \mathbb{A}^2(\mathbb{F}_{q})$. 
\end{enumerate}
\end{corollary}

This paper is organized as follows. 
In Section \ref{sec: monomial GAPN}, we present basic properties of monomial GAPN functions and provide two examples of $p$-exceptional monomial GAPN functions (see Proposition~\ref{ex: exceptionality}). 
In Section~\ref{Proof of Thm. 1.6}, we prove the main theorems. 
In Section~\ref{sec: examples}, we give examples of polynomials $\psi_d$ for various positive integers $d$ (see Examples~\ref{ex: psid gcd(d,p-1) ne 1}--\ref{ex: psid case(3)}). 

\section{Properties of monomial GAPN functions}
\label{sec: monomial GAPN}

Let $p$ be a prime number, and let $d$ and $n$ be positive integers. 
Put $f_d(x) \coloneqq x^{d} \in \mathbb{F}_{p}[x]$. 
For each polynomial $f \in \mathbb{F}_p[x]$, the function $\mathbb{F}_{p^n} \to \mathbb{F}_{p^n}$, $a \mapsto f(a)$ is denoted by the same notation $f$. 
For each non-negative integer $m$, let $\sum_{s \geq 0} c_s(m) p^s$ denote the $p$-adic expansion of $m$, and define its \textit{$p$-weight} $w_p(m)$ by $w_p(m) \coloneqq \sum_{s \geq 0} c_s(m)$. 
The \textit{algebraic degree} $d^{\circ}(f)$ of a non-zero polynomial function $f(x) = \sum_{m\geq0} c_m x^m$ is defined by $d^{\circ} (f) \coloneqq \max \{ w_p(m) \mid c_m \ne 0 \}$. 
Note that $d^{\circ}(f_d) = w_p(d)$. 

For any $a \in \mathbb{F}_{p^n}^{\times}$, the equalities  
\begin{align} \label{eq: GDafd}
GD_a f_d (x) 
= \sum_{j \in \mathbb{F}_p} (x + ja)^d 
= a^d \sum_{j \in \mathbb{F}_p} \left( \dfrac{x}{a} + j \right)^d
= a^d GD_1 f_d \left( \dfrac{x}{a} \right)
\end{align}
hold. 
Thus, for any $a \in \mathbb{F}_{p^n}^{\times}$ and any $b \in \mathbb{F}_{p^n}$, the following are equivalent: 
$GD_af_d(x) = b$; 
$GD_1f_d(x/a) = b/a^d$. 
Therefore, we have only to verify the case $a=1$ in Definition \ref{def: GAPN}. 
For simplicity, put 
\begin{align*}
GD_d (x) \coloneqq GD_1 f_d(x) = \sum_{j \in \mathbb{F}_p} (x + j)^d.
\end{align*} 

\begin{lemma} \label{lem: monomial GAPN}
Under the above notation, the following are equivalent: 
\begin{enumerate}[(i)]
\item the monomial $f_d(x) = x^d$ is a GAPN function on $\mathbb{F}_{p^n}$; 
\item 
the relation $\alpha - \beta \in \mathbb{F}_p$ holds for any pair of $\alpha \in \mathbb{F}_{p^n}$ and $\beta \in \mathbb{F}_{p^n}$ with $GD_d(\alpha) = GD_d (\beta)$. 
\end{enumerate}
\end{lemma}

\begin{proof}
The statement (i) holds if and only if $\# S(b) \leq p$ for any $b \in \mathbb{F}_{p^n}$, where $S(b) \coloneqq \{ x \in \mathbb{F}_{p^n} \mid GD_d(x) = b \}$. 
Note that $\{ x_0 + j \mid j \in \mathbb{F}_p \} \subseteq S(b)$ for any $b \in \mathbb{F}_{p^n}$ and any $x_0 \in S(b)$. 
Thus, the inequality $\# S(b) \leq p$ holds for any $b \in \mathbb{F}_{p^n}$ if and only if the statement (ii) holds. 
\end{proof}

For convenience, put $GD_0(x) \coloneqq 0 \in \mathbb{F}_p[x]$. 
For any positive integer $d$, the equalities 
\begin{align} \label{eq: GDd}
GD_d(x) 
= \sum_{j \in \mathbb{F}_p} (x+j)^d 
= \sum_{j \in \mathbb{F}_p} \sum_{k=0}^{d} \binom{d}{k} j^k x^{d-k} 
= \sum_{k \in S_p(d)} \binom{d}{k} GD_k(0) x^{d-k} 
\end{align}
hold, where 
$S_p(d) \coloneqq \{ k \in [0,d] \mid \mbox{$\binom{d}{k} \ne 0$ in $\mathbb{F}_p$} \}$ and $0^0 \coloneqq 1$. 
Lucas's theorem shows that 
\begin{align} \label{Sp(d)}
S_p(d) = \Set{ k \in [0,d] \mid \mbox{$c_s(k) \leq c_s(d)$ for any $s \geq 0$} }.
\end{align}
For any non-negative integer $k$, the equalities 
\begin{align} \label{eq: GDd(0)}
GD_k(0) = \sum_{j \in \mathbb{F}_p} j^k 
= \left\{ 
\begin{array}{ll}
-1 & (k \ne 0 \ and \ (p-1) \mid k), \\
0 & (otherwise) \\
\end{array}
\right.
\end{align}
hold. 
Note that $GD_k(0)=GD_{w_p(k)}(0)$ for any non-negative integer $k$ since $k \equiv w_p(k) \mod (p-1)$. 

\begin{proposition}[A generalization of {\cite[Proposition 4]{Kuroda2020}}] \label{prop:algebraic degree}
Let $n$ be a positive integer with $n \geq 2$.
\begin{enumerate}[(i)]
\item 
If $w_p(d) \leq p-2$ (resp.\ $w_p(d) = p-1$, resp.\ $w_p(d) \geq p$), then $GD_d(x) = 0 $ (resp.\ $GD_d(x)=-1$, resp.\ $GD_d(x) \not \in \mathbb{F}_p$). 

\item 
If $w_p(d) \leq p-1$, then $f_d(x)=x^d$ is not a GAPN function on $\mathbb{F}_{p^n}$. 

\item 
If $\gcd (d, p-1) \ne 1$, then $f_d(x)=x^d$ is not a GAPN function on $\mathbb{F}_{p^n}$. 
\end{enumerate}
\end{proposition}

\begin{proof}
\begin{enumerate}[(i)]
\setlength{\parindent}{10pt} 
\setlength{\parskip}{0pt}
\setlength{\parsep}{0.5\baselineskip}

\item 
Assume that $w_p(d) \leq p-2$. 
The equality (\ref{Sp(d)}) shows that $w_p(k) \leq w_p(d) \leq p-2$ for any $k \in S_p(d)$. 
Thus, the equalities (\ref{eq: GDd(0)}) show that 
$GD_k (0) = GD_{w_p(k)} (0) = 0$ for any $k \in S_p(d)$. 
Therefore, the equalities (\ref{eq: GDd}) show that $GD_d(x) = 0$. 

Assume that $w_p(d) = p-1$. 
The equality (\ref{Sp(d)}) shows that $w_p(k) < w_p(d) = p-1$ for any $k \in S_p(d) \setminus \{ d \}$, which implies that $GD_k(x) = 0$ for any $k \in S_p(d) \setminus \{ d \}$. 
Thus, the equalities (\ref{eq: GDd}) and (\ref{eq: GDd(0)}) show that $GD_d(x) = GD_d(0) = GD_{p-1}(0) = -1$. 

Assume that $w_p(d) \geq p$. 
Put $r \coloneqq w_p(d) - p + 1 \geq 1$. 
Then we may take an integer $k \in S_p(d)$ so that $w_p(k) = w_p(d) - r = p-1$. 
Since $GD_k(0) = -1$, the equalities (\ref{eq: GDd}) show that the polynomial $GD_d(x)$ has the non-zero term $-\binom{d}{k} x^{d-k}$, which implies that $GD_d(x) \not \in \mathbb{F}_p$. 

\item 
Proposition~\ref{prop:algebraic degree} (i) shows that if $w_p(d) \leq p-1$, then $GD_d(x)$ is constant. 
In that case, Lemma~\ref{lem: monomial GAPN} shows that $f_d$ is not a GAPN function on $\mathbb{F}_{p^n}$.

\item 
Let $a$ be a generator of the multiplicative group $\mathbb{F}_p^{\times}$.  
Put $t \coloneqq (p-1) / \gcd (d, p-1)$. 
Then $a^t \ne 1$ since $\gcd (d, p-1) \ne 1$. 
Since $a^{td} = 1$ and $GD_{\alpha} f_d(x) = GD_d(x)$ for any $\alpha \in \mathbb{F}_p^{\times}$,
it follows from the equalities \eqref{eq: GDafd}, by replacing $a$ with $a^{-t}$, that
\begin{align} \label{eq: gcd(d,p-1) ne 1}
GD_d(a^t x)
= a^{td} \, GD_{a^{-t}} f_d(x)
= GD_d(x).
\end{align}
Thus, for a fixed element $x_0 \in \mathbb{F}_{p^n} \setminus \mathbb{F}_p$, the equality $GD_d(a^t x_0) = GD_d(x_0)$ and the relation $a^t x_0 - x_0 \not \in \mathbb{F}_p$ hold since $a^t \ne 1$. 
Therefore, Lemma \ref{lem: monomial GAPN} shows that $f_d$ is not a GAPN function on $\mathbb{F}_{p^n}$. 
\qedhere
\end{enumerate}
\end{proof}

\begin{example} \label{GDd wp(d)=p}
Proposition~\ref{prop:algebraic degree} (i) explicitly determines $GD_d(x)$ in the case $w_p(d) \leq p-1$. 
Let us also compute $GD_d(x)$ in the case $w_p(d)=p$.
The equalities (\ref{eq: GDd})--(\ref{eq: GDd(0)}) and Lucas's theorem show that 
\begin{gather*}
GD_d(x) - GD_d(0)
= - \sum_{\substack{ k \in S_p(d) \\ w_p(k)=p-1 }} \binom{d}{k} x^{d-k} 
= - \sum_{\substack{ s \geq 0 \\ c_s(d)>0 }} \binom{d}{d-p^s} x^{p^s} 
= - \sum_{s \geq 0} c_s(d) x^{p^s} \\
\mbox{and} \quad 
GD_d(0) = GD_1 (0) =  \left\{ 
\begin{array}{ll}
1 & (p = 2), \\
0 & (p \geq 3), \\
\end{array}
\right.
\end{gather*}
which implies that $GD_d(x) - GD_d(0)$ is a linearized polynomial over $\mathbb{F}_p$. 
\end{example}

Proposition~\ref{prop:algebraic degree} (ii)--(iii) show that the monomial $f_d(x)=x^d$ is not a $p$-exceptional GAPN function if either $w_p(d) \leq p-1$ or $\gcd(d,p-1)\ne1$. 
If $w_p(d) = p$, then $\gcd(d, p-1) = 1$, and the monomial $f_d(x) = x^d$ is a $p$-exceptional GAPN function (see Theorem~\ref{thm: 2-excep} and Table~\ref{table:p-exceptional GAPN} (1)). 
While the case (1) in Table~\ref{table:p-exceptional GAPN} is previously known, only partial results are known in the cases (2)--(3) (see Remark~\ref{rem: table1}). 
Proposition~\ref{ex: exceptionality} below completes Table~\ref{table:p-exceptional GAPN} (see Remark~\ref{rem: conditions for (2)-(3)}).

\begin{proposition} \label{ex: exceptionality}
Let $p$ be an odd prime number, and let $i$, $k_1$ and $k_2$ be positive integers with $1 \leq k_1, k_2 \leq p-1$. 
Put $d_2 \coloneqq k_2 p^i + k_1$ and $d_3 \coloneqq k_2 p^{2i} + (p-1) p^i + k_1$. 
\begin{enumerate}[(i)]
\item 
The monomial $f_{d_2}(x) = x^{d_2}$ is a $p$-exceptional GAPN function if and only if $w_p(d_2) \geq p$ and $\gcd(d_2, p-1) = 1$. 
\item 
The monomial $f_{d_3}(x) = x^{d_3}$ is a $p$-exceptional GAPN function if $w_p(d_3) < 2(p-1)$ and $\gcd(d_3, p-1) = 1$. 
\end{enumerate}
\end{proposition}

\begin{proof} 
The only if part of the statement (i) follows from Proposition~\ref{prop:algebraic degree} (ii)--(iii).

Let us prove the if parts of the statements (i)--(ii). 
Put $u_2 \coloneqq k_1 + k_2 - (p-1)$ and $u_3 \coloneqq k_2p^i + k_1$. 
Then for any $j \in \{ 2, 3 \}$ and any positive integer $n$ such that $d_j < p^n$ and $\gcd(i,n)=\gcd(u_j, p^n-1)=1$, the monomial $f_{d_j}$ is a GAPN function on $\mathbb{F}_{p^n}$ 
(see Remarks~\ref{rem: conditions for (2)-(3)} and \ref{rem: table1} (ii)--(iii)). 
Let $e_q$ denote the order of $p$ in the multiplicative group $(\mathbb{Z} / q \mathbb{Z})^{\times}$ for each prime number $q$ with $q \ne p$. 
Then $e_q > 1$ if and only if $\gcd (q, p-1) = 1$. 
For each $j \in \{ 2, 3 \}$, put 
\begin{align*}
N_j \coloneqq &\Set{ n \in \mathbb{Z}_{>i(j-1)} | \mbox{$\gcd(i,n) = \gcd (u_j, p^n-1) =1$} } \\
= &\Set{ n \in \mathbb{Z}_{>i(j-1)} | \mbox{$\gcd(i,n) = 1$ and $e_q \nmid n$ for any prime factor $q$ of $u_j$ with $q \ne p$}  }. 
\end{align*}
For any $j \in \{ 2, 3 \}$, the following are equivalent: 
$N_j \ne \emptyset$; 
$e_q > 1$ for any prime factor $q$ of $u_j$ with $q \ne p$; 
$\gcd (u_j, p-1) = 1$; 
$\# N_j = \infty$. 
Since $\gcd(u_j, p-1) = \gcd(d_j, p-1) = 1$ for any $j \in \{ 2,3\}$, the monomials $f_{d_2}$ and $f_{d_3}$ are $p$-exceptional GAPN functions. 
\end{proof}

In the rest of this section, we collect several remarks concerning Table~\ref{table:p-exceptional GAPN}.

\begin{remark} \label{rem: conditions for (2)-(3)}
The notation is as in Proposition~\ref{ex: exceptionality}. 
\begin{enumerate}[(i)]
\item 
Since $w_p(d_2) = k_1+k_2$ and $\gcd(d_2, p-1) = \gcd(k_1+k_2, p-1)$, the conditions $k_1 + k_2 \geq p$ and $\gcd(k_1+k_2, p-1)=1$ in the case (2) in Table~\ref{table:p-exceptional GAPN} are equivalent to the conditions $w_p(d_2) \geq p$ and $\gcd(d_2, p-1)=1$. 
\item 
Since $w_p(d_3) = k_1+k_2 + p-1$ and $\gcd(d_3, p-1) = \gcd(k_1+k_2, p-1)$, the conditions $k_1 + k_2 < p-1$ and $\gcd(k_1+k_2, p-1)=1$ in the case (3) in Table~\ref{table:p-exceptional GAPN} are equivalent to the conditions $w_p(d_3) < 2(p-1)$ and $\gcd(d_3, p-1)=1$. 
\end{enumerate}
\end{remark}

\begin{remark} \label{d not equiv 0 mod p}
Let us explain why Table~\ref{table:p-exceptional GAPN} lists only exponents $d$ with $d \not\equiv 0 \mod p$. 
Although Table~\ref{table:p-exceptional GAPN} concerns only the case $p \geq 3$, the following argument is valid for any prime number $p$.
\begin{enumerate}[(i)]
\item 
Proposition~2.1 in \cite{KT2017} shows that for any positive integer $n$, the GAPN property on $\mathbb{F}_{p^n}$ is invariant under EA-equivalence. 

\item 
Proposition~1.2 in \cite{ZKPLZ2023} shows that for any positive integers $d$, $d'$ and $n$ with $d < p^n$ and $d' < p^n$, the monomials $f_d(x) = x^d$ and $f_{d'}(x) = x^{d'}$ are EA-equivalent on $\mathbb{F}_{p^n}$ if and only if $d \equiv p^\nu d' \mod (p^n-1)$ for some $0 \leq \nu \leq n-1$. 

\item 
Let $d$ be a positive integer. 
Assume that $d \equiv 0 \mod p$. 
Then $d = p^\nu d'$ for some positive integers $\nu$ and $d'$ with $d' \not \equiv 0 \mod p$.
Remark~\ref{d not equiv 0 mod p} (ii) shows that for any positive integer $n$ with $d<p^n$, the monomials $f_d(x) = x^d$ and $f_{d'}(x) = x^{d'}$ are EA-equivalent on $\mathbb{F}_{p^n}$. 
Remark~\ref{d not equiv 0 mod p} (i) shows that $f_d$ is a $p$-exceptional GAPN function if and only if $f_{d'}$ is a $p$-exceptional GAPN function. 
Thus, Table~\ref{table:p-exceptional GAPN} lists only exponents $d$ with $d \not\equiv 0 \mod p$.
\end{enumerate}
\end{remark}

\begin{remark} \label{rem: table1}
We summarize how the families of $p$-exceptional monomial GAPN functions in Table~\ref{table:p-exceptional GAPN} are related to the known results.
Let $d$ and $n$ be positive integers with $d < p^n$. 
\begin{enumerate}[(i)]
\setlength{\parindent}{10pt} 
\setlength{\parskip}{0pt}
\setlength{\parsep}{0.5\baselineskip}

\item 
We consider the case where $d$ is an exponent appearing in the case (1) in Table~\ref{table:p-exceptional GAPN}, that is, $d = p^{i_1} + \cdots + p^{i_{p-1}} + 1$ with certain conditions. 
Theorem~1 in \cite{OS2021} shows that $f_d$ is a GAPN function on $\mathbb{F}_{p^n}$ if and only if $\gcd(1+\sum_{j=1}^{p-1} z^{i_j}, z^n-1) = z-1$ in $\mathbb{F}_p[z]$. 

Theorem~1 in \cite{OS2021} contains the result of \cite[Theorem~8]{Kuroda2020}. 
Moreover, Proposition~9 in \cite{Kuroda2020} shows that $f_d$ is a $p$-exceptional GAPN function. 

\item 
We consider the case where $d$ is an exponent appearing in the case (2) in Table~\ref{table:p-exceptional GAPN}, that is, $d = k_2 p^{i} + k_1$ with certain conditions. 
Put $u \coloneqq k_1 + k_2 - (p-1)$. 
Theorem~2 in \cite{OS2023} shows that $f_d$ is a GAPN function on $\mathbb{F}_{p^n}$ if and only if $\gcd(i,n) = \gcd(u,p^n-1)=1$ in the case $n\geq 3$, and $\gcd(u, p-1)=1$ in the case $n=2$. 
Moreover, if $w_p(d) = p$, then $u=1$ and $d$ is of the form considered in Remark~\ref{rem: table1} (i). 
In that case, the following are equivalent: 
(a) $f_d$ is a GAPN function on $\mathbb{F}_{p^n}$; 
(b) $\gcd(i,n)=1$; 
(c) $\gcd(k_1 + k_2 z^i, z^n-1)=z-1$ in $\mathbb{F}_p[z]$. 
The equivalence of the statements (b)--(c) follows from the equality $k_1 + k_2 z^i = k_2 (z^i - 1)$ in $\mathbb{F}_p[z]$.

Theorem~2 in \cite{OS2023} contains the results of \cite[Theorem~3.3]{PXC2019}, \cite[Theorems~2.4 and 3.1]{ZHZ2018} and \cite[Theorem~2]{OS2021}.

\item 
We consider the case where $d$ is an exponent appearing in the case (3) in Table~\ref{table:p-exceptional GAPN}, that is, $d = k_2 p^{2i} + (p-1) p^i + k_1$ with certain conditions. 
Put $u \coloneqq k_2p^i + k_1$ and $d' \coloneqq k_2 p^{i} + (p-1) + k_1 p^{n-i}$. 
Then $f_d$ and $f_{d'}$ are EA-equivalent on $\mathbb{F}_{p^n}$ since $d \equiv p^{i}d' \mod (p^n-1)$ (see Remark~\ref{d not equiv 0 mod p} (ii)).
Theorem~3.1 in \cite{WWZ2022} shows that if $\gcd(i,n) = \gcd(u, p^n-1) = 1$, then $f_{d'}$ is a GAPN function on $\mathbb{F}_{p^n}$. 
Thus, the monomial $f_d$ is also a GAPN function on $\mathbb{F}_{p^n}$ under the same conditions. 

The special case where $k_1 = p-3$ and $k_2=1$ was considered in \cite[Theorem~3.1]{ZKPLZ2023}.
\end{enumerate}
\end{remark}

\section{The proofs of the main theorems}
\label{Proof of Thm. 1.6}

Let $p$ be a prime number, $i$ be a non-negative integer and $d$ be a positive integer. 
Let $\nu$ denote the non-negative integer such that $p^\nu \mid d$ and $p^{\nu+1} \nmid d$. 
Put $d' \coloneqq d / p^\nu$. 

\begin{definition} \label{def: psid}
\begin{enumerate}[(i)]
\item 
Put $\wp _i (x) \coloneqq x^{p^i} - x$ and $\wp (x) \coloneqq \wp_1 (x) = x^p - x$. 
\item 
Put $\phi_d (x, y) 
\coloneqq 
GD_d(x) - GD_d(y)
= \sum_{j \in \mathbb{F}_p} ( (x + j)^d - (y + j)^d )$. 
\item 
Put $\psi_d(x,y) \coloneqq \dfrac{\phi_d (x, y)}{(\wp(x-y))^{p^\nu}}$. 
\end{enumerate}
\end{definition}

Lemma \ref{lem: psid = 0} (ii) below shows that $\psi_d \in \mathbb{F}_p[x,y]$. 

\begin{lemma} \label{lem: psid = 0}
The notation is as above. 
\begin{enumerate}[(i)]
\item 
The following are equivalent: 
$w_p(d) < p$; 
$\phi_d = 0$; 
$\psi_d = 0$. 

\item 
The polynomial $\phi_d$ is divisible by the polynomial $\wp(x-y)$, and the equality $\phi_d(x,y) = (\phi_{d'}(x,y))^{p^\nu}$ holds.
In particular, the relations $\psi_d(x,y) = (\psi_{d'}(x,y))^{p^\nu} \in \mathbb{F}_p[x,y]$ hold. 
\end{enumerate}
\end{lemma}

\begin{proof}
\begin{enumerate}[(i)]
\item 
It follows from Proposition \ref{prop:algebraic degree} (i). 

\item 
Since $\phi_{d} (y+j,y) = GD_{d}(y+j) - GD_{d}(y) = 0$ for any $j \in \mathbb{F}_{p}$, the polynomial $\phi_{d}$ is divisible by $\prod_{j \in \mathbb{F}_p} (x - y - j) = (x-y)^p - (x-y) = \wp(x-y)$. 
Since $GD_{d} (x) = (GD_{d'}(x))^{p^\nu}$, the equality $\phi_d(x,y) = (\phi_{d'}(x,y))^{p^\nu}$ holds. 
\qedhere
\end{enumerate}
\end{proof}

Throughout this paper, two polynomials are said to be relatively prime if they have no common non-constant factor. 
In particular, if either of two relatively prime polynomials is zero, then the other must be a non-zero constant. 

\begin{lemma} \label{lem: def of psid}
The notation is as above. 
Assume that $\phi_d \ne 0$. 
\begin{enumerate}[(i)]
\item 
The polynomials $\phi_{d'}$ and $\psi_{d'}$ are reduced over $\overline{\mathbb{F}_p}$. 

\item 
The polynomials $\psi_d$ and $\wp(x-y)$ are relatively prime. 
\end{enumerate}
\end{lemma}

\begin{proof}
\begin{enumerate}[(i)]
\item 
It suffices to show that $\phi_{d'}$ is reduced over $\overline{\mathbb{F}_p}$ since $\phi_{d'} (x,y) = \wp(x-y) \psi_{d'}(x,y)$. 
Assume that there exists a non-constant polynomial $g \in \overline{\mathbb{F}_p} [x,y]$ such that $g^2 \mid \phi_{d'}$. 
Then the relations $g \mid \frac{\partial \phi_{d'}}{\partial x} = d' GD_{d'-1}(x)$ and $g \mid \frac{\partial \phi_{d'}}{\partial y} = -d' GD_{d'-1}(y)$ hold. 
Since $\phi_d \ne 0$, Lemma~\ref{lem: psid = 0} (i) shows that $w_p(d') = w_p(d) \geq p$. 
Proposition \ref{prop:algebraic degree} (i) shows that $d'GD_{d'-1}(x) \ne 0$ since $d' \not \equiv 0 \mod p$ and $w_p(d'-1) = w_p(d') - 1 \geq p-1$. 
Thus, the relations $g \in \overline{\mathbb{F}_p}[x] \cap \overline{\mathbb{F}_p}[y] = \overline{\mathbb{F}_p}$ hold, which is absurd. 

\item 
Lemma \ref{lem: def of psid} (i) shows that $\phi_{d'}$ is reduced over $\overline{\mathbb{F}_{p}}$, which implies that $\psi_{d'}$ and $\wp(x-y)$ are relatively prime since $\phi_{d'} (x,y) = \wp(x-y) \psi_{d'}(x,y)$. 
Thus, Lemma~\ref{lem: psid = 0} (ii) shows the assertion. 
\qedhere
\end{enumerate}
\end{proof}

For a subfield $K$ of $\overline{\mathbb{F}_p}$, we denote by $V(g)$ the closed subscheme of $\mathrm{Spec}\, K[x,y]$ defined by $g \in K[x,y]$ and denote by $V(g) (L)$ the set of $L$-rational points of $V(g)$, where $L \subseteq \overline{\mathbb{F}_p}$ is an extension field of $K$. 
Then the equality 
\begin{align} \label{eq: V(g)(L)}
V(g)(L) = \Set{ (\alpha, \beta) \in \mathbb{A}^2(L) | g(\alpha, \beta) = 0 } 
\end{align}
holds. 
More generally, for each $g \in \overline{\mathbb{F}_p}[x,y]$ and each subfield $L \subseteq \overline{\mathbb{F}_p}$, let $V(g)(L)$  denote the right hand side of the equality (\ref{eq: V(g)(L)}), even when $g \notin L[x,y]$. 
Throughout this paper, unless otherwise specified, a point $P$ of the scheme $V(g)$ over $K$ is always understood to be an $\overline{\mathbb{F}_p}$-rational point, that is, $P \in V(g)(\overline{\mathbb{F}_p})$.

Then a geometric characterization of monomial GAPN functions is given as follows. 

\begin{proposition} \label{prop:geometric characterization of GAPN functions} 
For any positive integer $n$, the following are equivalent: 
\begin{enumerate}[(i)]
\item 
$f_d(x) = x^d$ is a GAPN function on $\mathbb{F}_{p^n}$; 
\item 
$V(\phi_d) ( \mathbb{F}_{p^n} ) \subseteq 
\bigsqcup_{j \in \mathbb{F}_p} V(x - y + j) ( \mathbb{F}_{p^n} )$; 
\item 
$V(\psi_d) ( \mathbb{F}_{p^n} ) \subseteq 
\bigsqcup_{j \in \mathbb{F}_p} V(x - y + j) ( \mathbb{F}_{p^n} )$. 
\end{enumerate}
\end{proposition}

\begin{proof}
The equivalence of the statements (i)--(ii) (resp.\ (ii)--(iii)) follows from Lemma~\ref{lem: monomial GAPN} 
(resp.\ Definition~\ref{def: psid}).
\end{proof}

We now prove Theorem~\ref{thm:geom. irred. comp. induces not exceptional}, which is our first main result. 
It follows from Lemma~\ref{lem: def of psid}, Proposition~\ref{prop:geometric characterization of GAPN functions} and Proposition~\ref{prop: generalization of key proposition} below.
For convenience, we restate the theorem. 
Proposition~\ref{prop: generalization of key proposition} and its proof are given after the proof of the restated theorem. 

\begin{theorem}[See Theorem~\ref{thm:geom. irred. comp. induces not exceptional}]
\label{thm: main1 restatement} 
Let $p$ be a prime number and $d$ be a positive integer. 
Then the monomial $f_d(x) = x^d$ is a $p$-exceptional GAPN function if and only if the polynomial $\psi_d$ is non-zero and has no absolutely irreducible factor defined over $\mathbb{F}_p$.
\end{theorem}

\begin{proof} 
If $w_p(d) < p$, then Proposition~\ref{prop:algebraic degree} (ii) shows that $f_d(x) = x^d$ is not a $p$-exceptional GAPN function, and Lemma~\ref{lem: psid = 0} (i) shows that $\psi_d = 0$. 
Thus, we may assume that $w_p(d) \geq p$. 
Lemma~\ref{lem: psid = 0} (i) shows that $\phi_d \ne 0$.  
Lemma~\ref{lem: def of psid} (ii) implies that Proposition~\ref{prop: generalization of key proposition} below is applicable with $s = t = 1$, $f(x,y) = \wp(x-y)$ and $h = \psi_d$. 
Thus, Proposition~\ref{prop:geometric characterization of GAPN functions} and Proposition~\ref{prop: generalization of key proposition} below show that the following are equivalent: 
$f_d(x) = x^d$ is a $p$-exceptional GAPN function; 
there exist infinitely many positive integers $n$ such that $V(\psi_d) (\mathbb{F}_{p^n}) \subseteq \bigsqcup_{j \in \mathbb{F}_p} V(x - y + j) (\mathbb{F}_{p^n})$; 
$\psi_d$ is non-zero and has no absolutely irreducible factor defined over $\mathbb{F}_p$. 
\end{proof}

For each subfield $L \subseteq \overline{\mathbb{F}_p}$ and each polynomial $h \in L[x,y]$, let $V(h)_{\mathrm{red}}$ denote the reduction of the scheme $V(h)$ over $L$, and let $\mathrm{Sing}(V(h)_{\mathrm{red}})(L)$ denote the set of $L$-rational points of the singular
locus $\mathrm{Sing} (V(h)_{\mathrm{red}})$ of $V(h)_{\mathrm{red}}$. 
Put 
\begin{align*}
N(h) \coloneqq \min \Set{ n \in \mathbb{Z}_{>0} | \mathrm{Sing} (V(h)_{\mathrm{red}})(\overline{\mathbb{F}_p}) \subseteq V(h)(\mathbb{F}_{p^n}) }. 
\end{align*}

\begin{proposition} \label{prop: generalization of key proposition}
Let $s$ and $t$ be positive integers. 
Put $q \coloneqq p^t$. 
Let $f \in \overline{\mathbb{F}_p}[x,y]$ and $h \in \mathbb{F}_q[x,y]$. 
Put $r \coloneqq \gcd(t,N(h))$. 
Consider the following three conditions: 
\begin{enumerate}[(i)]
\item 
the polynomial $h$ is non-zero and has no absolutely irreducible factor defined over $\mathbb{F}_q$; 
\item 
there exist infinitely many positive integers $n$ such that $V(h) (\mathbb{F}_{q^n}) \subseteq \mathbb{A}^2(\mathbb{F}_{p^s})$; 
\item 
there exist infinitely many positive integers $n$ such that $V(h) (\mathbb{F}_{q^n}) \subseteq V(f) (\mathbb{F}_{q^n})$. 
\end{enumerate}
Then the following statements hold: 
\begin{enumerate}[(a)]
\item 
If $r \mid s$, then the condition (i) implies the condition (ii);
\item 
If $V(f) (\mathbb{F}_{p^s}) = \mathbb{A}^2 (\mathbb{F}_{p^s})$, then the condition (ii) implies the condition (iii);
\item 
If the polynomials $h$ and $f$ are relatively prime, then the condition (iii) implies the condition (i). 
\item 
The condition (ii) implies the condition (i).
\end{enumerate}
\end{proposition}

\begin{proof}
In the case $h \in \mathbb{F}_q$, 
the assertions are clear. 
Thus, we consider the case $h \not \in \mathbb{F}_q$.

\begin{enumerate}[(a)]
\setlength{\parindent}{10pt} 
\setlength{\parskip}{0pt}
\setlength{\parsep}{0.5\baselineskip}

\item 
Assume that $r \mid s$. 
Then Lemma~\ref{lem:key} below shows that the condition (i) implies the condition (ii). 

\item 
Assume that $V(f) (\mathbb{F}_{p^s}) = \mathbb{A}^2 (\mathbb{F}_{p^s})$. 
Let $n$ be a positive integer. 
Then the relations 
\begin{align*}
\mathbb{A}^2(\mathbb{F}_{p^s}) \cap \mathbb{A}^2(\mathbb{F}_{q^n}) 
= V(f) (\mathbb{F}_{p^s}) \cap \mathbb{A}^2(\mathbb{F}_{q^n}) 
\subseteq V(f) (\mathbb{F}_{q^n})
\end{align*}
show that the condition (ii) implies the condition (iii). 

\item 
Assume that the polynomials $h$ and $f$ are relatively prime and that the condition (iii) holds. 
Let us prove that the condition (i) holds. 
Assume that $h$ has an absolutely irreducible factor $g$ defined over $\mathbb{F}_q$. 
Since $g$ is an absolutely irreducible polynomial in $\mathbb{F}_{q^n}[x,y]$ for any positive integer $n$, Theorem~\ref{thm:lower bound of rational points} below implies that $q^{n/2} \leq \# V(g) (\mathbb{F}_{q^n})$ for any sufficiently large $n$. 
The condition (iii) shows that there exists a positive integer $m$ such that 
\begin{gather*}
V(g) \left( \mathbb{F}_{q^{m}} \right) \subseteq 
V(h) \left( \mathbb{F}_{q^{m}} \right) \subseteq 
V(f) \left( \mathbb{F}_{q^{m}} \right) \quad \mbox{and} \quad 
\deg f \deg g < q^{\frac{m}{2}} \leq \# V(g) \left( \mathbb{F}_{q^{m}} \right). 
\end{gather*}
Lemma~\ref{lem:upper bound of Fpn-rational points} below shows that $\# V(g) \left( \mathbb{F}_{q^{m}} \right) \leq \deg f \deg g $, which is absurd. 

\item 
Assume that $h$ has an absolutely irreducible factor $g$ defined over $\mathbb{F}_q$. 
Similarly to the proof of Proposition~\ref{prop: generalization of key proposition} (c), Theorem~\ref{thm:lower bound of rational points} below shows that $q^{n/2} \leq \# V(g) (\mathbb{F}_{q^n}) \leq \# V(h) (\mathbb{F}_{q^n})$ for any sufficiently large $n$. 
This contradicts the condition (ii). 
\qedhere
\end{enumerate}
\end{proof}

\begin{lemma} \label{lem:key}
Let $t$ be a positive integer. 
Put $q \coloneqq p^t$. 
Let $h \in \mathbb{F}_q[x,y] \setminus \{ 0 \}$. 
Put $r \coloneqq \gcd(t, N(h))$. 
Assume that the polynomial $h$ has no absolutely irreducible factor defined over $\mathbb{F}_q$. 
Then there exist infinitely many positive integers $n$ such that $V(h) (\mathbb{F}_{q^n}) = \mathrm{Sing}(V(h)_{\mathrm{red}})(\mathbb{F}_{q^n}) \subseteq \mathbb{A}^2( \mathbb{F}_{p^r} )$. 
\end{lemma}

\begin{proof}
Let $h = g_1 \cdots g_u$ be a prime factorization in $\mathbb{F}_q [x,y]$ and $g_i = g_{i, 1} \cdots g_{i, v_i}$ be a prime factorization in $\overline{\mathbb{F}_p} [x,y]$ for each $1 \leq i \leq u$. 
Put 
\begin{align*}
m_{i, j} \coloneqq \min \Set{ m \in \mathbb{Z}_{>0} | \mbox{$g_{i, j}$ is defined over $\mathbb{F}_{q^m}$} }
\end{align*}
for each $1 \leq i \leq u$ and each $1 \leq j \leq v_i$. 
Since $h$ has no absolutely irreducible factor defined over $\mathbb{F}_q$, the inequality $m_{i, j} > 1$ holds for any $1 \leq i \leq u$ and any $1 \leq j \leq v_i$. 
Then for any positive integer $n$ and any $1 \leq i \leq u$, the following are equivalent: 
(A) $g_i$ has no absolutely irreducible factor defined over $\mathbb{F}_{q^n}$; 
(B) $g_{i,j}$ is not defined over $\mathbb{F}_{q^n}$ for any $1 \leq j \leq v_i$; 
(C) the relation $m_{i, j} \nmid n$ holds for any $1 \leq j \leq v_i$. 
The equivalence of the statements (A)--(B) follows from the prime factorization $g_i = g_{i, 1} \cdots g_{i, v_i}$ in $\overline{\mathbb{F}_p}[x,y]$. 
The equivalence of the statements (B)--(C) follows from the minimality of $m_{i, j}$ since $\mathrm{Gal} (\mathbb{F}_{q^n} / \mathbb{F}_q) \cong \mathbb{Z} / n \mathbb{Z}$ for any positive integer $n$. 
Put 
\begin{align*}
T &\coloneqq \Set{ n \in \mathbb{Z}_{>0} | \mbox{$m_{i, j} \nmid n$ for any $1 \leq i \leq u$ and any $1 \leq j \leq v_i$} } \quad \mbox{and}\\
T_i &\coloneqq \Set{ 
n \in \mathbb{Z}_{>0} | 
\mbox{$g_i$ has no absolutely irreducible factor defined over $\mathbb{F}_{q^n}$} 
}  
\end{align*}
for each $1 \leq i \leq u$. 
The above equivalence of the statements (A)--(C) shows that $T = \bigcap_{1 \leq i \leq u} T_i$. 
Put 
\begin{align*}
N \coloneqq N(h) \quad \mbox{and} \quad 
T' \coloneqq \{ n \in T \mid \gcd(n, N/r) = 1 \}. 
\end{align*}
The equality $\# T' = \infty$ holds since the inequality $m_{i, j} > 1$ holds for any $1 \leq i \leq u$ and any $1 \leq j \leq v_i$. 
Thus, it suffices to show that 
\begin{align} \label{statment: key Prop}
\mbox{for any $n \in T'$, the relations $V(h) (\mathbb{F}_{q^n}) = \mathrm{Sing} (V(h)_{\mathrm{red}}) (\mathbb{F}_{q^n}) \subseteq \mathbb{A}^2( \mathbb{F}_{p^r} )$ hold.} 
\end{align}

Let us show the statement (\ref{statment: key Prop}). 
Let $n \in T'$. 
We have only to consider the case $V(h) (\mathbb{F}_{q^n}) \ne \emptyset$. 
Let $P = (\alpha, \beta) \in V(h) (\mathbb{F}_{q^n})$. 
Then $P \in V(g_i) ( \mathbb{F}_{q^n} )$ for some $1 \leq i \leq u$. 
Let $V$ be an irreducible component of $V(g_i)$ over $\mathbb{F}_{q^n}$ containing $P$. 
Then $V$ is a variety over $\mathbb{F}_{q^n}$ (see Definition~\ref{def: variety} for the definition of a variety over $K$). 
Since $n \in T_i$, Lemma \ref{lem: T subset S} below shows that $P$ is a singular point of $V$. 
Since the point $P$ is also a singular point of $V(h)_{\mathrm{red}}$, the relations $P \in \mathrm{Sing} (V(h)_{\mathrm{red}}) (\mathbb{F}_{q^n}) \subseteq V(h)(\mathbb{F}_{p^N})$ hold. 
Thus, the relations $\{ \alpha, \beta \} \subset \mathbb{F}_{q^n} \cap \mathbb{F}_{p^N} = \mathbb{F}_{p^{nt}} \cap \mathbb{F}_{p^N} = \mathbb{F}_{p^r}$ hold since $\gcd(n,N/r) = 1$ and $\gcd(t,N) = r$. 
Therefore, the statement (\ref{statment: key Prop}) holds. 
\end{proof}

\begin{definition}[See {\cite[Tag 020D]{stacks-project}}] \label{def: variety}
Let $K$ be a field. 
A \textit{variety over $K$} is a scheme $X$ over $K$ such that $X$ is integral and the structure morphism $X \to \mathrm{Spec}\, K$ is separated and of finite type. 
\end{definition}

\begin{lemma}[See {\cite[Tag 0CDW]{stacks-project}}] \label{lem: T subset S}
Let $K$ be a field. 
Let $X$ be a variety over $K$ that has a $K$-rational point at which $X$ is smooth. 
Then $X$ is geometrically integral over $K$, that is, $X$ is both geometrically reduced and geometrically irreducible over $K$. 
\end{lemma}

\begin{theorem}[The Hasse--Weil bound for possibly singular affine curves; see {\cite[Theorem~6.4.1]{FJ}} for example] 
\label{thm:lower bound of rational points}
Let $n$ be a positive integer. 
For any absolutely irreducible polynomial $g \in \mathbb{F}_{p^n} [x, y]$ of degree $d'$, 
the following inequalities hold: 
\begin{align*}
p^n + 1 - (d' - 1)(d' - 2) p^{\frac{n}{2}} - d'
\leq 
\# V(g) (\mathbb{F}_{p^n}) 
\leq 
p^n + 1 + (d' - 1) (d' - 2) p^{\frac{n}{2}} .
\end{align*}
\end{theorem}

\begin{lemma} \label{lem:upper bound of Fpn-rational points}
Let $n$ be a positive integer, let $f \in \overline{\mathbb{F}_p}[x,y] \setminus \{0\}$ and let $g \in \overline{\mathbb{F}_p}[x,y]$. 
Assume that 
\begin{enumerate}[(i)]
\item 
the polynomials $f$ and $g$ are relatively prime, and 
\item 
$V(g)(\mathbb{F}_{p^n}) \subseteq V(f) (\mathbb{F}_{p^n})$. 
\end{enumerate}
Then $\# V(g) ( \mathbb{F}_{p^n} ) \leq \deg f \deg g$.
\end{lemma}

\begin{proof}
In the case $f \in \overline{\mathbb{F}_p}^{\times}$, the assumption (ii) shows that $V(g)(\mathbb{F}_{p^n}) = \emptyset$ and $g \ne 0$, which proves the desired inequality. 
We consider the case $f \not \in \overline{\mathbb{F}_p}$. 
Then the assumption (i) shows that $g \ne 0$. 
In the case $g \in \overline{\mathbb{F}_p}^{\times}$, the desired inequality holds. 
Thus, we consider the case $g \not \in \overline{\mathbb{F}_p}$.
The assumption (ii) shows that 
\begin{align*}
V(g) \left( \mathbb{F}_{p^n} \right) 
= V(g)( \mathbb{F}_{p^n} ) \cap V(f) ( \mathbb{F}_{p^n} )
= \left( V (g) \cap V(f) \right)( \mathbb{F}_{p^n} ).
\end{align*}
The assumption (i) and B\'ezout's theorem imply that $\# ( V (g) \cap V(f) )( \mathbb{F}_{p^n} ) \leq \deg f\deg g$. 
Thus, the desired inequality holds. 
\end{proof}

\begin{remark} \label{rem: Prop.3.6}
The proof of Theorem~\ref{thm: main1 restatement} relies on the equivalence of the conditions (i) and (iii) in Proposition~\ref{prop: generalization of key proposition}. 
Theorem~\ref{thm: main 2} below follows from the equivalence of the conditions (i)--(ii) in Proposition~\ref{prop: generalization of key proposition}. 
We note that for any $h \in \overline{\mathbb{F}_p}[x,y] \setminus \{ 0\}$, there exists a polynomial $f \in \mathbb{F}_p[x,y]$ such that the assumptions in Proposition~\ref{prop: generalization of key proposition} (b)--(c) hold (see Lemma~\ref{lem:existence of f}). 
\end{remark}

\begin{theorem}
\label{thm: main 2}
Let $p$ be a prime number and $t$ be a positive integer. 
Put $q \coloneqq p^t$. 
Let $h \in \mathbb{F}_q[x,y] \setminus \{ 0 \}$. 
Put $r \coloneqq \gcd(t, N(h))$. 
Then the following are equivalent: 
\begin{enumerate}[(i)]
\item 
the polynomial $h$ has no absolutely irreducible factor defined over $\mathbb{F}_q$; 
\item 
for any positive integer $s$ with $r \mid s$, there exist infinitely many positive integers $n$ such that $V(h) (\mathbb{F}_{q^n}) \subseteq \mathbb{A}^2(\mathbb{F}_{p^s})$; 
\item 
there exist infinitely many positive integers $n$ such that $V(h) (\mathbb{F}_{q^n}) \subseteq \mathbb{A}^2(\mathbb{F}_{p^r})$; 
\item 
for some positive integer $s$ with $r \mid s$, there exist infinitely many positive integers $n$ such that $V(h) (\mathbb{F}_{q^n}) \subseteq \mathbb{A}^2(\mathbb{F}_{p^s})$. 
\end{enumerate}
\end{theorem}

\begin{proof}
It follows from Proposition~\ref{prop: generalization of key proposition}. 
\end{proof}

Theorem~\ref{thm: main 2} is our second main result. 
Moreover, Corollary~\ref{thm:geom. char. of red. poly.} is a direct consequence of Theorem~\ref{thm: main 2}. 
The following example illustrates Theorem~\ref{thm: main 2} in the case $r = t = N(h)$. 

\begin{example}
Let $t$ be a positive integer. 
Put $q \coloneqq p^t$. 
Let $a$ be a generator of the multiplicative group $\mathbb{F}_q^{\times}$, and let $b \in \mathbb{F}_{q^2} \setminus \mathbb{F}_q$. 
Put 
\begin{align*}
h(x,y) \coloneqq (x - a - by)(x - a - b^qy) 
= (x-a)^2 - (b+b^q)(x-a)y + b^{q+1}y^2. 
\end{align*}
Then $h \in \mathbb{F}_{q}[x,y]$ since $(b+b^q)^q = b^q + b^{q^2} = b+b^q$ and $(b^{q+1})^q = b^{q^2+q} = b^{q+1}$. 
The polynomial $h$ is irreducible over $\mathbb{F}_q$ and has the absolutely irreducible factors $x - a - by$ and $x - a - b^qy$ defined over $\mathbb{F}_{q^2}$, which implies that $h$ is reduced over $\overline{\mathbb{F}_p}$. 
The curve $V(h)$ has a unique singular point $(a, 0)$, that is, $\mathrm{Sing}(V(h))(\overline{\mathbb{F}_{p}}) = \{ (a,0) \}$. 
Since $\mathbb{F}_p(a) = \mathbb{F}_q$, the equalities $r = t = N(h)$ hold. 
For any positive integer $n$, the equality 
\begin{align*}
V(h)(\mathbb{F}_{q^n}) 
= \left\{ 
\begin{array}{ll}
\{ (a,0) \} & (\mbox{$n$ is odd}), \\
\{ (a,0) \} \sqcup \{ (a + c \beta, \beta) \mid \mbox{$c \in \{ b, b^q\}$ and $\beta \in \mathbb{F}_{q^n}^{\times}$} \} & (\mbox{$n$ is even}) \\
\end{array}
\right.
\end{align*}
holds, which implies that $V(h)(\mathbb{F}_{q^n}) = \mathrm{Sing}(V(h))(\mathbb{F}_{q^n}) \subset \mathbb{A}^2(\mathbb{F}_q)$ for any odd $n$.  
\end{example}

\begin{lemma} \label{lem:existence of f}
Let $s$ be a positive integer. 
Let $h \in \overline{\mathbb{F}_p}[x,y] \setminus \{ 0 \}$. 
Then there exists a polynomial $f \in \mathbb{F}_p[x,y]$ such that 
\begin{enumerate}[(i)]
\item 
$V(f) (\mathbb{F}_{p^s}) = \mathbb{A}^2 (\mathbb{F}_{p^s})$, and 
\item 
the polynomials $h$ and $f$ are relatively prime. 
\end{enumerate}
\end{lemma}

\begin{proof}
In the case $h \in \overline{\mathbb{F}_p}^{\times}$, the assertion holds for $f=0$. 
Thus, we consider the case $h \not \in \overline{\mathbb{F}_p}$.
Replacing $h$ by the product of all its Galois conjugates over $\mathbb{F}_p$, we may assume that $h \in \mathbb{F}_p[x,y]$.
Put $I \coloneqq (\wp_s(x), \wp_s(y)) \subset \mathbb{F}_p[x,y]$. 
Since the equality (i) holds for any $f \in I$, it suffices to show that there exists a polynomial $f \in I$ such that $h$ and $f$ are relatively prime. 
Let $h = g_1 \cdots g_u$ be a prime factorization in $\mathbb{F}_p [x,y]$. 
Since $\wp_s(x)$ and $\wp_s(y)$ are relatively prime, the relation $I \not \subseteq (g_i)$ holds for any $1 \leq i \leq u$. 
The prime avoidance lemma shows that $I \not \subseteq \bigcup_{i=1}^{u} (g_i)$. 
Let $f \in I \setminus \bigcup_{i=1}^{u} (g_i)$. 
Since $g_i \nmid f$ for any $1 \leq i \leq u$, the polynomials $h$ and $f$ are relatively prime. 
\end{proof}

\section{Examples of polynomials $\psi_d$}
\label{sec: examples}

Let $p$ be a prime number. 
In this section, we give several examples of polynomials $\psi_d \in \mathbb{F}_p[x,y]$ for various positive integers $d$. 
Recall that $\psi_d = 0$ if and only if $w_p(d) < p$ (see Lemma~\ref{lem: psid = 0} (i)). 
In the case where $w_p(d)>p$ and $ \gcd (d, p-1) \ne 1$, the polynomial $\psi_d$ has an absolutely irreducible factor defined over $\mathbb{F}_p$ (see Example~\ref{ex: psid gcd(d,p-1) ne 1}). 
On the other hand, Theorem~\ref{thm:geom. irred. comp. induces not exceptional} shows that the polynomial $\psi_d$ has no absolutely irreducible factor defined over $\mathbb{F}_p$ for any $d$ appearing in the cases (1)--(3) in Table \ref{table:p-exceptional GAPN}. 
In Examples~\ref{ex: psid wp(d)=p}--\ref{ex: psid case(3)} below, we verify directly that the polynomial $\psi_d$ has no absolutely irreducible factor defined over $\mathbb{F}_p$ for such integers $d$. 
To this end, we establish explicit formulas for $GD_d(x)$.

Lemma~\ref{lem: GDd 1} below provides a convenient method for computing $GD_d(x)$, which is used in Examples~\ref{ex: psid case(2)}--\ref{ex: psid case(3)} below.
Put 
\begin{align*}
GD_n^{(s)} (x) \coloneqq \sum_{j \in \mathbb{F}_p} 
j^s (x + j)^n
\end{align*}
for each positive integer $n$ and each non-negative integer $s$. 
Recall that $0^0 \coloneqq 1$. 

\begin{lemma} \label{lem: GDd 1}
Let $n$ be a positive integer, and let $i$ and $k$ be non-negative integers. 
Then the following equalities hold: 
\begin{align}
GD_{n}^{(k)} (x) &= \sum_{s = 0}^k \binom{k}{s} (-x)^{k-s} GD_{n+s} (x) \qquad \mbox{and} 
 \label{eq: GD n(s)}
\\
GD_{p^i n+k} (x) 
&= \sum_{s = 0}^k \binom{k}{s} x^{k-s} \left( GD_n^{(s)} (x) \right)^{p^i} 
= \sum_{s = 0}^k \binom{k}{s} \left( - \wp_i(x) \right)^{k-s} \left( GD_{n+s} (x) \right)^{p^i}. 
\label{eq: GD n+k}
\end{align}
\end{lemma}

\begin{proof}
Let us show the equality (\ref{eq: GD n(s)}) and the first equality in the equalities (\ref{eq: GD n+k}). 
For any positive integer $N$, the following equalities hold in $\mathbb{F}_p [x, y, z]$: 
\begin{align*}
\sum_{j \in \mathbb{F}_p} (y+z+j)^{k} (x+j)^{N}
&= \sum_{j \in \mathbb{F}_p} \sum_{s=0}^{k} \binom{k}{s} (y+j)^s z^{k-s} (x+j)^{N} 
\\
&= \sum_{s=0}^{k} \binom{k}{s} z^{k-s} 
\sum_{j \in \mathbb{F}_p} (y+j)^s (x+j)^N .
\end{align*}
The equality (\ref{eq: GD n(s)}) follows from the case where $(y,z) = (x,-x)$ and $N = n$. 
Since $GD_{p^in}^{(s)} (x) = ( GD_n^{(s)} (x))^{p^i} $, the first equality in the equalities (\ref{eq: GD n+k}) follows from the case where $(y,z) = (0,x)$ and $N = p^in$. 

Let us show the second equality in the equalities (\ref{eq: GD n+k}). 
The equality (\ref{eq: GD n(s)}) shows that 
\begin{align*}
\sum_{s = 0}^k \binom{k}{s} x^{k-s} \left( GD_n^{(s)} (x) \right)^{p^i} 
&= \sum_{s = 0}^k \binom{k}{s} x^{k-s} 
\sum_{t=0}^{s} \binom{s}{t} \left( -x^{p^i} \right)^{s-t} \left( GD_{n+t}(x) \right)^{p^i} 
\\
&
= \sum_{t=0}^{k} \sum_{s=t}^k \binom{k}{s} \binom{s}{t} x^{k-s} (-x^{p^i} )^{s-t} \left( GD_{n+t}(x) \right)^{p^i}. 
\end{align*}
Put $C_t \coloneqq \sum_{s=t}^k \binom{k}{s} \binom{s}{t} x^{k-s} (-x^{p^i} )^{s-t}$ for each $0 \leq t \leq k$. 
It suffices to show that $C_t = \binom{k}{t}(- \wp_i(x))^{k-t}$ for any $0 \leq t \leq k$. 
Since $- \wp_i(x) = x - x^{p^i}$ and $\binom{k}{t} \binom{k-t}{s-t} = \binom{k}{s} \binom{s}{t}$ for any $0 \leq t \leq k$ and any $t \leq s \leq k$, the equalities 
\begin{align*}
\binom{k}{t}(- \wp_i(x))^{k-t} 
&= \sum_{u = 0}^{k-t} \binom{k}{t} \binom{k-t}{u} x^{k-t-u} \left( - x^{p^i} \right)^u 
= \sum_{s = t}^{k} \binom{k}{t} \binom{k-t}{s-t} x^{k-s} \left( - x^{p^i} \right)^{s-t} 
\\
&= \sum_{s = t}^{k} \binom{k}{s} \binom{s}{t} x^{k-s} \left( - x^{p^i} \right)^{s-t} 
= C_t 
\end{align*}
hold for any $0 \leq t \leq k$. 
\end{proof}

\begin{example} \label{ex: psid gcd(d,p-1) ne 1}
Let $d$ be a positive integer with 
$w_p(d) > p$ and $\gcd( d , p-1 ) \ne 1$. 
Let $a$ be a generator of the multiplicative group $\mathbb{F}_p^{\times}$ and $s$ be a positive integer with $s \mid \gcd( d , p-1 )$ and $s \ne 1$. 
Put $t_s \coloneqq (p-1)/s$. 
Then the equality $GD_d (a^{t_s} x) = GD_d (x)$ holds similarly to the equalities (\ref{eq: gcd(d,p-1) ne 1}). 
Definition~\ref{def: psid} shows that $(y-a^{t_s} x) \mid \psi_d$ since $a^{t_s} \ne 1$. 
Thus, the polynomial $\psi_d$ has the absolutely irreducible factor $y-a^{t_s} x$ defined over $\mathbb{F}_p$. 
\end{example}

\begin{example} \label{ex: psid wp(d)=p}
Let $d$ be a positive integer with 
$w_p(d) = p$ and $d \not \equiv 0 \mod p$. 
Then $d = p^{i_1} + \cdots + p^{i_{p-1}} + 1$ with $i_1 \geq \cdots \geq i_{p-1} \geq 0$ and $i_1 \ne 0$. 
If $p \geq 3$, then $d$ is of the form considered in the case (1) in Table~\ref{table:p-exceptional GAPN}. 
Example~\ref{GDd wp(d)=p} shows that 
\begin{align*}
\psi_d (x, y) 
= \dfrac{GD_d (x) - GD_d (y)}{\wp(x-y)} 
= \dfrac{GD_d (x-y) - GD_d(0)}{\wp(x-y)} . 
\end{align*}
Thus, any irreducible factor of $\psi_d$ in $\overline{\mathbb{F}_p}[x,y]$ is of the form $x-y + \alpha$ for some $\alpha \in \overline{\mathbb{F}_p}$. 
Lemma~\ref{lem: def of psid} (ii) shows that $\alpha \not \in \mathbb{F}_p$. 
Therefore, the polynomial $\psi_d$ has no absolutely irreducible factor defined over $\mathbb{F}_p$. 
\end{example}

\begin{example} \label{ex: psid case(2)}
Let $d$ be an exponent appearing in the case (2) in Table~\ref{table:p-exceptional GAPN}, that is, $d = k_2 p^i + k_1$ with certain conditions. 
Since $p \leq k_1 + k_2 < 2 (p-1)$, the equalities (\ref{eq: GD n+k}) and Proposition~\ref{prop:algebraic degree} (i) show that 
\begin{align*}
GD_d(x) 
&= \sum_{s=0}^{k_1} \binom{k_1}{s} \left( - \wp_i(x) \right)^{k_1-s} \left( GD_{k_2+s}(x) \right)^{p^i} 
\\
&= \binom{k_1}{p-1-k_2} \left( - \wp_i(x)  \right)^{k_1+k_2-(p-1)} (-1)^{p^i} 
= (-1)^{u+1} \binom{k_1}{u} \left( \wp_i(x) \right)^u, 
\end{align*}
where $u \coloneqq k_1+k_2-(p-1)$. 
Since $\gcd(u,p-1) = \gcd(k_1+k_2,p-1) = 1$, Lemma~\ref{lem: good GDd} below shows that $\psi_{d}$ has no absolutely irreducible factor defined over $\mathbb{F}_p$. 
\end{example}

\begin{example} \label{ex: psid case(3)}
Let $d$ be an exponent appearing in the case (3) in Table~\ref{table:p-exceptional GAPN}, that is, $d = k_2 p^{2i} + (p-1) p^i + k_1$ with certain conditions. 
Since $d = p^i (k_2 p^{i} + p-1) + k_1$, the equalities (\ref{eq: GD n+k}) show that 
\begin{align*}
GD_d(x)
= \sum_{s=0}^{k_1} \binom{k_1}{s} \left( - \wp_i(x) \right)^{k_1-s} \left( GD_{k_2 p^{i} + p-1+s}(x) \right)^{p^i} .
\end{align*}
Since $k_1 + k_2 < p-1$, the equalities (\ref{eq: GD n+k}) and Proposition~\ref{prop:algebraic degree} (i) show that 
\begin{align*}
GD_{k_2 p^{i} + p-1+s}(x)
&= \sum_{t=0}^{p-1+s} \binom{p-1+s}{t} \left( - \wp_i(x) \right)^{p-1+s-t} \left( GD_{k_2 + t}(x) \right)^{p^i} 
\\
&= \binom{p-1+s}{p-1-k_2} \left( - \wp_i(x) \right)^{s+k_2} (-1)^{p^i} 
\end{align*}
for any $0 \leq s \leq k_1$, and Lucas's theorem shows that $\binom{p-1+s}{p-1-k_2} = \binom{s-1}{p-1-k_2} = 0$ for any $1 \leq s \leq k_1$. 
Since 
\begin{align*}
\binom{p-1}{k_2} = \frac{(p-1)(p-2) \cdots (p-k_2)}{k_2(k_2-1) \cdots 1} = (-1)^{k_2} 
\quad \mbox{in $\mathbb{F}_p$}
\end{align*} 
for any $1 \leq k_2 \leq p-1$, the equalities 
\begin{align*}
GD_d(x) = \left( - \wp_i(x) \right)^{k_1} 
\left(- \binom{p-1}{p-1-k_2} \left( - \wp_i(x) \right)^{k_2} \right)^{p^i} 
= (-1)^{k_1+1} \left( \wp_i(x) \right)^{u}
\end{align*}
hold, where $u \coloneqq k_2 p^i + k_1$. 
Since $\gcd(u, p-1) = \gcd(k_1+k_2,p-1)=1$, Lemma~\ref{lem: good GDd} below shows that $\psi_{d}$ has no absolutely irreducible factor defined over $\mathbb{F}_p$. 
\end{example}

\begin{lemma} \label{lem: good GDd}
Let $d$ be a positive integer. 
Assume that $GD_d(x) = c_1 (\wp_i (x))^u + c_2$ for some $c_1 \in \mathbb{F}_p^{\times}$, some $c_2 \in \mathbb{F}_p$, and some positive integers $i$ and $u$. 
Then the following are equivalent: 
\begin{enumerate}[(i)]
\item 
$\gcd (u, p-1) = 1$; 
\item 
the polynomial $\psi_{d}$ has no absolutely irreducible factor defined over $\mathbb{F}_p$. 
\end{enumerate}
\end{lemma}

\begin{proof}
Let $\nu$ and $d'$ be as defined at the beginning of Section~\ref{Proof of Thm. 1.6}. 
Since $GD_d(x) = (GD_{d'}(x))^{p^\nu}$, the equality $(GD_{d'}(x) - c_2)^{p^\nu} = c_1 (\wp_i(x))^u$ holds. 
By comparing the multiplicities of the factor $x$ on both sides, we obtain the relation $p^{\nu} \mid u$. 
Put $u' \coloneqq u/p^\nu$. 
Then $GD_{d'}(x) = c_1 (\wp_i(x))^{u'}+c_2$ and $\gcd(u', p-1) = \gcd(u, p-1)$. 
Lemma~\ref{lem: psid = 0} (ii) shows that $\psi_d(x,y) = (\psi_{d'}(x,y))^{p^\nu}$. 
Thus, it suffices to consider the case $d \not \equiv 0 \mod p$. 
Lemma~\ref{lem: def of psid} (i) shows that $\psi_{d}$ is reduced over $\overline{\mathbb{F}_p}$. 
Thus, the equalities
\begin{align*}
\psi_{d} (x, y) 
= \dfrac{GD_{d} (x) - GD_{d} (y)}{\wp(x-y)} 
= c_1
\prod_{\alpha \in \mathbb{F}_{p^{i}} \setminus \mathbb{F}_p} (x - y - \alpha)
\prod_{\beta \in U(u)} \left( \wp_i(x) - \beta \wp_i(y) \right)
\end{align*}
hold, where $U(u) \coloneqq \{ \beta \in \overline{\mathbb{F}_{p}} \mid  \beta^{u} = 1 \mbox{ and } \beta \ne 1 \}$. 
It suffices to show that the following are equivalent: 
(a) $\gcd(u, p-1)=1$; 
(b) $U(u) \cap \mathbb{F}_{p} = \emptyset$; 
(c) the polynomial $\psi_{d}$ has no absolutely irreducible factor defined over $\mathbb{F}_p$. 

The equivalence of the statements (a)--(b) follows from the equality $ \mathbb{F}_p^{\times} = \{ \gamma \in \overline{\mathbb{F}_{p}} \mid \gamma^{p-1} = 1 \}$. 
Let us show the equivalence of the statements (b)--(c). 
Assume that there exists an element $\beta \in U(u) \cap \mathbb{F}_p$. 
Then $\wp_i(x) - \beta \wp_i(y) = \wp_i(x - \beta y)$ has the absolutely irreducible factor $x - \beta y$ defined over $\mathbb{F}_p$. 
Assume that $U(u) \cap \mathbb{F}_{p} = \emptyset$. 
Let us show that the statement (c) holds. 
Assume that $\psi_d$ has an absolutely irreducible factor $g$ defined over $\mathbb{F}_p$. 
Then $g \mid (\wp_i(x) - \beta \wp_i(y))$ for some $\beta \in U(u)$. 
By applying the Frobenius automorphism to the coefficients of these polynomials in $\overline{\mathbb{F}_p}[x,y]$, we obtain the relation $g \mid (\wp_i(x) - \beta^p \wp_i(y))$ since $g \in \mathbb{F}_p[x,y]$. 
Since $\beta^p \ne \beta$, the factor $g$ divides both $\wp_i(x)$ and $\wp_i(y)$, which leads to a contradiction. 
\end{proof}

\section*{Acknowledgments}

The first author is supported by JSPS KAKENHI Grant Number JP22K13906.
The second author is supported by JSPS KAKENHI Grant Numbers JP21K03179, JP24KK0253, JP26K06770. 

\bibliographystyle{amsplain}
\bibliography{bibfile}

\end{document}